\documentclass[leqno,12pt]{amsart}
\usepackage{amsmath,amsfonts,amsthm,mathrsfs,amssymb}
\usepackage[all]{xy}
\usepackage{enumerate}
\usepackage{graphicx,epsfig}
\usepackage{hyperref}
\usepackage{mathtools}
\usepackage{bm}
\usepackage{tikz}
\usepackage{pgf}
\usepackage{epsfig}

\usetikzlibrary{matrix}

\newtheorem{theorem}{Theorem}[section]
\newtheorem*{proposition*}{Proposition}
\newtheorem{proposition}[theorem]{Proposition}
\newtheorem{corollary}[theorem]{Corollary}
\newtheorem{lemma}[theorem]{Lemma}
\newtheorem{example}[theorem]{Example}
\newtheorem{remark}[theorem]{Remark}
\newtheorem{definition}[theorem]{Definition}

\newcommand\Cc{\mathcal C}
\DeclareMathOperator{\Supp}{Supp}
\DeclareMathOperator{\Proj}{Proj}
\DeclareMathOperator{\reg}{reg}
\DeclareMathOperator{\Tor}{Tor}
\DeclareMathOperator{\Hom}{Hom}
\DeclareMathOperator{\ann}{ann}
\DeclareMathOperator{\tdeg}{tdeg}

\DeclareMathOperator{\indeg}{indeg}
\title{Multigraded  Tor and local cohomology}

\author{Marc Chardin}
\address{Institut de Mathématiques de Jussieu, CNRS \& Sorbonne Université, France}
\email{marc.chardin@imj-prg.fr, Marc.CHARDIN@cnrs.fr}

\author{Rafael Holanda}
\address{Departamento de Matemática, Universidade Federal da Paraíba - 58051-900, João Pessoa, PB, Brazil}
\email{rfh@academico.ufpb.br, rf.holanda@gmail.com}


\begin{document}

\maketitle


\begin{abstract}
Notions of Castelnuovo-Mumford regularity and of $a^*$ invariant were extended from standard graded algebras to the toric setting. 

We here focus our attention on the standard multigraded case, which corresponds to a product of $k$ projective spaces. A natural notion for a $\mathbb Z^k$-graded module is its support : degrees in which it is not zero. A stabilized version of it is adding $-\mathbb N^k$, in order for the complement (vanishing region) to be stable by addition of $\mathbb N^k$. 

Cohomology of twists of a sheaf on a product of projective spaces, provided by a graded module, are given by local cohomologies with respect the product $B$ of the ideals $B_i$ generated by the $k$ sets of variables. 


Our results shed some light on a central issue, the relation between shifts in graded free resolution and cohomology vanishing : it shows that stabilized support of cohomology with respect to $B$ corresponds to the union of stabilized supports for cohomologies in the $B_i$'s, while shifts in (some of the) graded free resolutions are inside the intersection of these stabilized supports. A one-to-one correspondence between stabilized supports of Tor modules and of local cohomologies with respect to the sum of the $B_i$'s is also established.

We then derive a consequence on linear resolutions for truncations of a graded module. 
\end{abstract}

\section{Introduction}

Let $S$ be a commutative ring and $R$ a finitely generated standard $\mathbb Z ^k$-graded polynomial extension. In other words $R$ has $k\geq 1$ finite sets of variables, the $i$-th set having degree the $i$-th canonical generator of $\mathbb Z^k$. 

A graded $R$-module $M$ determines a quasi-coherent sheaf $\mathcal F$ on the scheme $\Proj (R)$ which is the product (over the spectrum of $S$) of projective spaces (again over the spectrum of $S$) corresponding to the $k$ sets of variables. Any quasi-coherent sheaf arises this way similarly as in the case of a single projective space. Sheaf cohomology of twists of $\mathcal F$ could be given in terms of graded components of the local cohomology of $M$ with respect to the product $B$ of the ideals $B_1,\ldots ,B_k$ generated by the $k$ sets of variables.

In the case $k=1$, the notion of Castelnuovo-Mumford regularity provides a key relation between shifts in graded free resolutions of $M$ and the degrees where local (or sheaf) cohomologies vanish. This notion could also be defined by considering truncations of the module (keeping only elements of degree at least some given integer) and asking for generation in a single degree and maximal shifts at each step in a graded resolution increasing at most by one. 

In a series of works, first motivated by toric geometry (see \cite{MS}), the notion of Castelnuovo-Mumford regularity was extended and studied for the multigraded setting, in general requiring $S$ to be a field and $M$ to be finitely generated. Although the properties do not extend perfectly, pretty sharp results have been obtained and provide a more and more precise picture of the situation; in the standard multigraded setting, see for instance \cite{BCS} for a recent refinement and references to previous works.

Besides the notion of Castelnuovo-Mumford regularity, a variation of this notion sometimes called the $a^*$-invariant could as well be defined both in terms of shifts in graded free resolutions or in terms of local cohomology. It is bounded above by the regularity and the difference is at most the number of variables; it gives a sharper estimate of the regularity of the Hilbert function, but do not share all the nice features of regularity. 

In this text, we consider the natural extension of the notion of $a^*$-invariant by considering the following notion already present in the work of Huy T\`ai H\`a \cite{Ha}:
$$
\mathbb C_B (M)^*:=\cup_i \Supp_{\mathbb Z^k}(H^i_B(M))^*
$$
where the support of a graded module is the set of degrees for which the corresponding component is not $0$. The star of a set $E\subseteq \mathbb Z^k$, $E^*:=\{ e-\mu\ \vert\ \mu \in \mathbb N^k\}$, is the smallest set containing $E$ such that its complement is stable under the addition of any element in $\mathbb N ^k$ (see Lemma \ref{stabcomp}).

In order to understand cohomology with support in $B$, which could be hard to determine even in very simple examples (Remark \ref{example} presents one such case), one available tool is a spectral sequence (see \cite{Lyu} or our construction in \ref{marcss})
$$
\bigoplus_{1\leq i_1<\cdots <i_p\leq k}H^{p-q-1}_{B_{i_1}+\cdots +B_{i_p}}(M)\Rightarrow H^q_B (M).
$$
It shows that $\mathbb C_{B}(M)\subseteq\bigcup_{i_1,\dots ,i_p} \mathbb C_{B_{i_1}+\cdots +B_{i_p}}(M)$. The cohomology with respect to the sum of some $B_j$'s is much easier to analyze, as it is cohomology with support in ideals generated by variables. 

A first surprise  was that in terms of supports, and after stabilization of the complement by the star operation, the union of the supports on the right equals $\mathbb C_B (M)^*$. More precisely, defining $\mathbb C_I (M)$ similarly as above for any graded ideal $I$:

\begin{theorem}{\rm(Theorem \ref{cohB})}
Let $M$ be a  graded $R$-module, then
$$
\bigcup_i \mathbb C_{B_{i}}(M)^*= \bigcup_{i_1,\dots ,i_p} \mathbb C_{B_{i_1}+\cdots +B_{i_p}}(M)^*=\mathbb C_{B}(M)^*.
$$
\end{theorem}

And also, explaining and completing the left equality:
$$
\mathbb C_{B_{i_1}+\dots +B_{i_p}}(M)^*\subseteq \mathbb C_{B_{i_1}}(M)^*\cap \cdots \cap \mathbb C_{B_{i_p}}(M)^*.
$$

Hence $\mathbb C_{B}(M)^*$ is fully determined by the supports of the cohomology with respect to the $B_i$'s. This should help to determine $\mathbb C_{B}(M)^*$ much more easily.

The second important, although less surprising, fact is that the sets $\mathbb C_{B_{i_1}+\dots +B_{i_p}}(M)^*$ are fully determined by the support of corresponding Tor modules, and vice-versa, as detailed in Theorem \ref{C=T}. This follows from the key case of $ \mathfrak{m}:=B_1+\cdots +B_k$:

\begin{theorem}{\rm(Theorem \ref{basicincl})}
Let $M$ be a  graded $R$-module, then
$$
\bigcup_i \Supp_{\mathbb{Z}^k} (\Tor_i^R(M,S))^*=(n_1,\ldots ,n_k)+\mathbb C_\mathfrak{m}(M)^*.
$$
\end{theorem}

These two results give a quite precise picture of the difference between the support of Tor modules, that reflects shifts in free resolutions (see Lemma \ref{TorFR} and following comments), and  cohomology with respect to $B$, a central question in many recent works on multigraded regularity : 
$$
\bigcup_i \Supp_{\mathbb{Z}^k} (\Tor_i^R(M,S))^*\subseteq (n_1,\ldots ,n_k)+\bigcap _i C_{B_{i}}(M)^*\ \hbox{ while } \mathbb C_{B}(M)^*=\bigcup_i \mathbb C_{B_{i}}(M)^*.
$$

This also supports the idea, already developed in a work of the first named author with Nicolas Botbol \cite{BC} : one should not restrict the attention to only the ideal $B$, but consider as well cohomologies supported on other graded ideals that are naturally linked to the study.

In our text, we do not ask the ring to satisfy any property (besides commutativity) nor finiteness hypotheses for the module, in order to give as much flexibility in the use as it could be. However, it is of course crucial to have in mind that the main tool to be able to assert that $\mathbb C_{B}(M)^*\not= \mathbb Z^k$, or to show that the modules $\Tor_i^R(M,S)$ have finite support,  is to require $S$ to be Noetherian and $M$ to be finitely generated. 

Whenever $\mu\not\in \mathbb C_{B}(M)^*$, we show that, in degrees slightly bigger then $\mu$, the truncation of the module is generated in a single degree and shifts have total degree increasing at most by one at each step (hence the truncated module has a linear resolution whenever the base ring is a field). We state this result in terms of the regularity in $R$ for the total degree grading (a $\mathbb Z^k$-graded module is also graded for the $\mathbb Z$-grading given by total degree):


\begin{proposition}{\rm(Proposition \ref{linrestrunc})}
Let $M$ be a graded  $R$-module and $\mu\not\in \mathbb C_B(M)^*$. 

Then for any $t\in \mu +(n_1-1,\ldots ,n_k-1)+\mathbb N^k$, $M_{t +\mathbb N^k}$ has a regularity $\vert t\vert$.
\end{proposition}

\section{Multigraded support}

Let $R$ be a commutative unitary ring and $I=(f_1,\ldots ,f_s)$ a finitely generated ideal of $R$. Given an $R$-module $M$, the \v{C}ech cohomology modules $H^i_I(M)$ of $M$ supported on $I$ are the cohomologies of the \v{C}ech complex $\check\Cc^\bullet_\mathbf{f}(M)$ of $M$ with respect to the finite tuple $\mathbf{f}=(f_1,\ldots ,f_s)$. These modules only depend upon $M$ and the radical of $I$, up to isomorphism; they coincide with the derived functors of $H^0_I (-)$ applied to $M$ whenever $R$ is Noetherian, or $I$ is generated by a regular sequence, or if $R$ is a polynomial ring over another ring and $I$ is a monomial ideal -- this will be the context in use in most of this text.

We now setup notations for working in standard multigraded polynomial rings. 
Let $S$ be a commutative unitary ring, $k\geq1$ an integer and denote the standard $\mathbb Z^k$-graded polynomial ring by $R=S[X_{1,1},...,X_{1,n_1},...,X_{k,1},...,X_{k,n_k}]$. For each set of variables $\mathbf{X}_i=\{X_{i,1},\ldots,X_{i,n_i}\}$ write $B_i$ for the ideal generated by this $i$-th set of variables, $B=B_1\cap\ldots\cap B_k=B_1\cdots B_k$ and $\mathfrak{m}=B_1+\ldots+B_k$. We write $\mathbb Z^k =\oplus_{i=1}^{k}\mathbb Z e_i$, where the $i$-th canonical generator of $\mathbb Z^k$ corresponds to the degree of the set of variables generating $B_i$. Unless another grading is specified, a graded $R$-module is a $\mathbb Z^k$-graded $R$-module. The degree will be written $\deg$ and $\tdeg$ will denote the total degree, they respectively belong to $\mathbb Z^k$ and $\mathbb Z$. Also $\mathbb N =\mathbb Z_{\geq 0}$ and if $G$ is an additive group and $E,F$ two subsets of $G$, $E+F=\{ e+f\ \vert\ e\in E, f\in F\}$, $-E:=\{ -e\ \vert\ e\in E\}$ and $E-F=E+(-F)$; to simplify notations if $e\in G$,  $e+F:=\{ e\} +F$. These conventions and notations will be in use for all the text and we define

\begin{definition}\label{supdef}
The \emph{support} of a graded $R$-module $M$ is $$\Supp_{\mathbb Z^k}(M):=\{\gamma\in \mathbb Z^k :M_\gamma\neq0\}.$$
Also, given a homogeneous finitely generated ideal $I$, we set $$\mathbb C_I^i(M):=\Supp_{\mathbb Z^k}(H^i_I(M)) \ \mbox{and} \ \mathbb C_I(M):=\bigcup_{i}\mathbb C_I^i(M).$$
\end{definition}

For a graded $R$-module $M$ and for $\mu\in\mathbb Z^k$ we define the $R$-module $M(-\mu)$ as $M$ with grading defined by $M(-\mu )_\nu :=M_{\mu -\nu}$.

\begin{lemma}\label{shiftlemma}
If $M$ is a graded $R$-module, then $\Supp_{\mathbb Z^k} (M(-\mu))=\Supp_{\mathbb Z^k}(M)+\mu$. 
\end{lemma}

\begin{proof}
Indeed $\nu \in \Supp_{\mathbb Z^k} (M(-\mu ))$ if and only if $\nu -\mu\in \Supp_{\mathbb Z^k} (M)$, equivalently $\nu \in \Supp_{\mathbb Z^k} (M)+\mu$.
\end{proof}

Shifts appearing in local cohomology with respect to the sum of some of the $B_i$'s will appear in many places, and we introduce the notation:
$$
a_{i_1,\cdots ,i_p}:=-\sum_{j=1}^{p}n_{i_j}e_{i_j}\in \mathbb Z^k\quad {\rm and}\quad a:=a_{1,\ldots ,k}=-(n_1,\ldots ,n_k).
$$

\begin{example}\label{suppex}
For $1\leq i_1<\cdots <i_p\leq k$, $\mathbb C^i_{B_{i_1}+...+B_{i_p}}(R)=\emptyset$ unless $i=n_{i_1}+\cdots +n_{i_p}$ and
 $$\mathbb C^{n_{i_1}+\cdots +n_{i_p}}_{B_{i_1}+...+B_{i_p}}(R)=a_{i_1,\cdots ,i_p}+\prod_{j=1}^k (-1)^{\varepsilon_j}\mathbb N,$$
 with $\varepsilon_j :=1$ if $j\in \{ i_1,\ldots ,i_p\}$ and $\varepsilon_j :=0$ else.
 
In particular, $\mathbb C^i_\mathfrak{m}(R)=\emptyset$ for $i\not= n_1+\cdots +n_k$ and
$$\mathbb C_\mathfrak{m}(R)=\mathbb C_\mathfrak{m}^{n_1+\cdots +n_k}(R)
=a-\mathbb{N}^k.$$

\end{example}

The following example illustrates supports.

\begin{example}\label{supportk=2ex}
By taking $k=2$ and $n_1=3$ and $n_2=5$ in Example \ref{suppex}, we have the following regions $\mathbb C_\mathfrak{m}(R)=(-3,-5)-\mathbb{N}^2, \mathbb C_{B_1}(R)=(-3,0)+(-\mathbb{N})\times\mathbb{N}$ and $\mathbb C_{B_2}(R)=(0,-5)+\mathbb{N}\times (-\mathbb{N})$. 

 \begin{figure}[!h]
 \includegraphics[width=8cm]{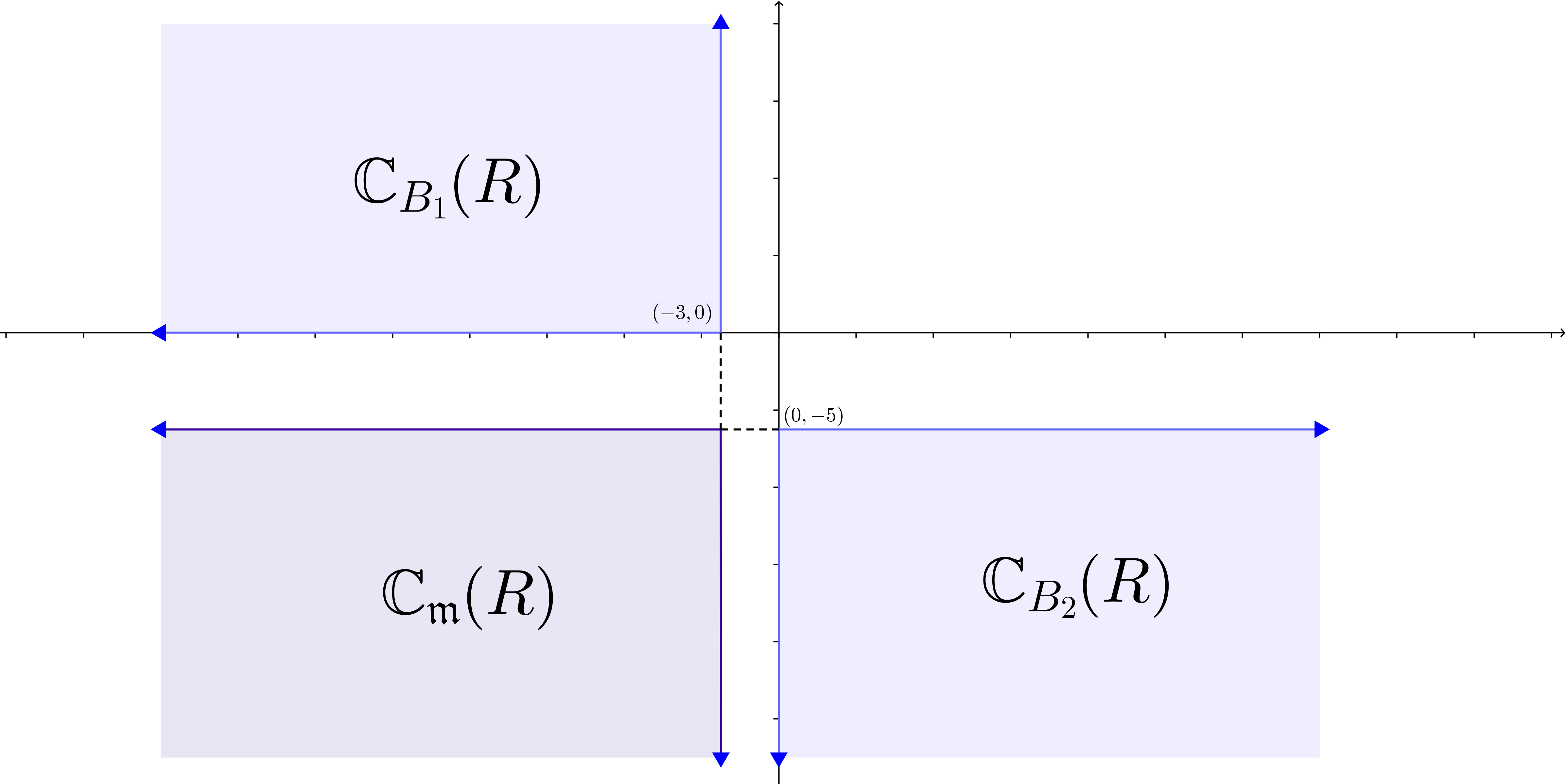}
 \centering
 \end{figure}

\end{example}
As detailed in the work of Eisenbud, Mustata and Stillman \cite{EMS} whenever the base ring is a field, cohomology with respect to $B$ of a graded $R$-module $M$ is providing the sheaf cohomology of the corresponding sheaf $\mathcal{F}$ on $\mathbb{P}:=\mathbb{P}^{n_1-1}\times\cdots\times \mathbb{P}^{n_k-1}$ (where projective spaces and products are taken over the spectrum of $S$):
$$
H^i(\mathbb{P},\mathcal{F}(\mu ))\simeq H^{i+1}_{B}(M)_\mu ,\quad \forall i>0,\quad \forall \mu\in \mathbb{Z}^k,
$$
and there is a natural short exact sequence of graded modules
$$
\xymatrix{
0\ar[r]&H^{0}_{B}(M)\ar[r]& M\ar[r]&\oplus_\mu H^0(\mathbb{P},\mathcal{F}(\mu ))\ar[r]&H^{1}_{B}(M)\ar[r]& 0.\\
}
$$

The following result is essential as it shows that the regions where cohomology vanishes that we will investigate are non trivial in many cases; recall also that, for a same module, local cohomology is invariant under arbitrary base change.

\begin{theorem}\label{NoethBounded}
Let $M$ be a finitely generated graded $R$-module. If $S$ is Noetherian, 

(i) $\Supp (\Tor_i^R(M,S))$ is a finite set, for every $i$,

(ii) there exists $\mu$ in $\mathbb{Z}^k$ such that 
$\mathbb C_{\mathfrak m} (M)\subseteq \mu -\mathbb{N}^k$,

(iii) 
there exists $\nu$ in $\mathbb{Z}^k$ such that 
$\mathbb C_B (M)\cap (\nu +\mathbb{N}^k )=\emptyset$.
\end{theorem}

\begin{proof}
For (i), recall that these Tor modules are finitely generated $S$-modules. Items
(ii) and (iii) follow from \cite[Theorem 4.14]{BC} : for (ii) by  Examples \ref{suppex} and because any finite union of shifts of $-\mathbb N^k$ is contained in $\mu -\mathbb N^k$ for some $\mu$, and for (ii) as there exists $\nu\in\reg_B(M)$ and, for any such $\nu$, $\mathbb C_B(M)\cap(\nu+\mathbb N^k)=\emptyset$.
\end{proof}

\section{Tor and local cohomology modules}

\subsection{An auxiliary result}

We first recall the following classical result that is a key ingredient in our first main result.

\begin{theorem}\label{CohModRes}
Let $I$ be a homogeneous ideal generated by a regular sequence ${\mathbf f}$ of length $r$ and $M$ a graded $R$-module. If $F_\bullet$ is a graded free $R$-resolution of $M$, for any $p$, there exists a degree zero graded isomorphism
$$
H^p_I (M)\simeq H_{r-p}(H^r_I (F_\bullet)).
$$
\end{theorem}

\begin{proof}
The second pages of the two spectral sequences from the double complex $\check\Cc^\bullet_{\mathbf f}(F_\bullet)$  are concentrated respectively on a line (as $H^i_I (F_\bullet)=0$ unless $i=r$) and on a column (as localization is flat) and are isomorphic to the right and left hand sides of the claimed isomorphism.
\end{proof}

As consequence of Theorem \ref{CohModRes} and \cite[Lemma 6.4.7]{Bot} we obtain the following spectral sequence.

\begin{corollary}\label{marcss}
Let $M$ be a graded $R$-module. There exists a spectral sequence of graded modules,
$$
\bigoplus_{1\leq i_1<\cdots <i_p\leq k}H^{p-q-1}_{B_{i_1}+\cdots +B_{i_p}}(M)\Rightarrow H^q_B (M).
$$
\end{corollary}

\begin{proof}
Let $F_\bullet$ be a graded free resolution of $M$ and $B=(f_1,\ldots ,f_s)$. The double complex $\check\Cc^\bullet_{\mathbf f} (F_\bullet)$
$$\xymatrix@=1em{
&0\ar[d]&0\ar[d]&0\ar[d]& \\
 \cdots\ar[r] & \mathcal C^0_{\mathbf f} (F_2 )\ar[r]\ar[d] & \mathcal C^0_{\mathbf f} (F_1 )\ar[r]\ar[d] & \mathcal C^0_{\mathbf f} (F_0)\ar[r]\ar[d] & 0
\\
 \cdots\ar[r] & \mathcal C^1_{\mathbf f}(F_2)\ar[r]\ar[d] & \mathcal C^1_{\mathbf f}(F_1)\ar[r]\ar[d] & \mathcal C^1_{\mathbf f}(F_0)\ar[r]\ar[d] & 0
\\
 &\vdots& \vdots&\vdots&
\\
}$$ 
yields two spectral sequences that converge to filtrations of the same graded module $H$. Taking homologies in the horizontal direction first, we get a spectral sequence $'E$ such that $'E_2^{0,-j}=H^j_B(M)$ and $'E_2^{-i,-j}=0$ whenever $i\neq0$; hence $H^j_B(M)\simeq H^j$ for all $j\geq0$. On the other hand, by \cite[Lemma 6.4.7]{Bot}, the other spectral sequence $E$ is such that $$E_2^{-i,-j}=H_i\left(\bigoplus_{\substack{1\leq i_1<...<i_p\leq k\\n_{i_1}+...+n_{i_p}=j+p-1}}H^{n_{i_1}+...+n_{i_p}}_{B_{i_1}+...+B_{i_p}}(F_\bullet)\right)\simeq\bigoplus_{1\leq i_1<...<i_p\leq k}H^{p+(j-i)-1}_{B_{i_1}+...+B_{i_p}}(M)$$ and $E_2^{-i,-j}\Rightarrow_i H^{i-j}\simeq H^{i-j}_B(M)$, whence the result.
\end{proof}

It is worth mentioning that the spectral sequence above has similarities with the one of Lyubeznik's \cite{Lyu}, although it has second page terms on a diagonal equal to the ones on a diagonal in the first page in Lyubeznik's approach.
\subsection{Stable sets}

It is convenient to introduce a notation for multigraded support of Tor modules, and for some shifts of these that are tightly connected to support of local cohomology, as we will show.

\begin{definition}
Given a graded $R$-module $M$ and $i_1,\dots ,i_p$ distinct elements in $\{ 1,\dots ,k\}$
$$
\mathbb T_j^{i_1,\dots ,i_p} (M):=\Supp_{\mathbb{Z}^k} (\Tor_j^R(M,R/(B_{i_1}+\dots +B_{i_p}))$$
and ${\mathbb T}^{i_1,\dots ,i_p}(M):=\cup_j \mathbb T_j^{i_1,\dots ,i_p} (M)$. Also,
$$
\hat{\mathbb T}_j^{i_1,\dots ,i_p} (M):=\mathbb T_j^{i_1,\dots ,i_p} (M)+a_{i_1,\ldots ,i_p}\ {\rm and}\ \hat{\mathbb T}^{i_1,\dots ,i_p} (M):={\mathbb T}^{i_1,\dots ,i_p} (M)+a_{i_1,\ldots ,i_p}.
$$

For simplicity,
$$
\mathbb T_j (M):=\mathbb T_j^{1,\dots ,k} (M)=\Supp_{\mathbb{Z}^k} (\Tor_j^R(M,S))\quad{\rm and}\quad \mathbb T (M):=\cup_j \mathbb T_j(M),
$$
and similarly $\hat{\mathbb T}_j (M):=\mathbb T_j (M)+a$ and $\hat{\mathbb T} (M):=\mathbb T (M)+a$.
\end{definition}

\begin{lemma}\label{stabcomp}
Let $E\subset \mathbb{Z}^k$. The smallest set containing $E$ such that its complement is stable under the addition of $\mathbb{N}^k$ is $E-\mathbb{N}^k$.
\end{lemma}

\begin{proof}
Let $E^*:=E-\mathbb{N}^k$. If $e\not\in E^*$, $e+n$ is not in $E^*$ for all $n\in \mathbb{N}^k$ as otherwise $e+n=e'-n'$ with $e'\in E$ and $n'\in \mathbb{N}^k$, the equality $e=e'-n-n'$ then contradicting the fact that $e\not\in E^*$. Now $E^*$ is minimal as for $E\subseteq E'\subseteq E^*$, with the complement of $E'$ stable, if $e\in E^*\setminus E'$, then $e+n\in E\subseteq E'$ for some $n\in\mathbb{N}^k$,  but also $e+n\notin E'$ due to the stability of the complement of $E'$, thus a contradiction.
\end{proof}

From now on, we denote by $E^*$ the set $E-\mathbb{N}^k$ as in the proof above.

Recall from \cite[Lemma 3.12 i)]{BC} a link between Tor modules and graded free resolutions :

\begin{lemma}\label{TorFR}
Let $M$ be a graded $R$ module such that $\tdeg (m)\geq u$ for some $u\in \mathbb Z$ and every $m\in M$. Then $M$ admits a graded free $R$-resolution $F_\bullet$ such that $F_i$ is generated by elements with degrees sitting in $\cup_{p\leq i}\mathbb T_p(M)$. 
\end{lemma}

Any minimal graded free resolution (no syzygy module has a superfluous generator) satisfy this property. If $S$ is Noetherian and $M$ is finitely generated,  minimal graded free resolutions exist.   Any element in $\mathbb T_i$ appear as a shift in $F_i$ for any graded free $R$-resolution $F_\bullet$ of $M$. One can replace the union by the only set $\mathbb T_i$ if $S$ is a field.

We first show that support of cohomology and support of Tor modules are in one to one correspondence, if one considers the union of all in both cases and stabilize the complement by addition of $-\mathbb{N}^k$. The first and main case is the one relative to the ideal $\mathfrak{m}$.

\begin{theorem}\label{basicincl}
Let $M$ be a graded module.  Then 
$$
\mathbb C_\mathfrak{m}(M)^*=\mathbb T (M)+\mathbb C_\mathfrak{m}(R)=\hat{\mathbb T} (M)^*.
$$
\end{theorem}

\begin{proof}
Recall that $\mathbb C_\mathfrak{m}(R)=a-\mathbb{N}^k$, by Example \ref{suppex}. Thus $\mathbb T (M)+\mathbb C_\mathfrak{m}(R)=\hat{\mathbb T} (M) -\mathbb{N}^k=\hat{\mathbb T} (M)^*$.

We first assume that there exists $u\in \mathbb Z$ such that $\tdeg (m)\geq u$ for any $m\in M$ and let $F_\bullet$ be a graded free $R$-resolution of $M$ as in  Lemma \ref{TorFR}. Then,
$$
H^i_\mathfrak{m}(M)\simeq H_{d-i}(H^d_\mathfrak{m} (F_\bullet)),
$$
for all $i\geq0$ by Theorem \ref{CohModRes}, where $d:=n_1+\ldots+n_k$ is the number of variables of $R$. This shows that $\mathbb C_\mathfrak{m}^i(M)\subseteq \bigcup_{p\leq d-i}\mathbb T_{p}(M)+\mathbb C_\mathfrak{m}(R)$, hence $\mathbb C_\mathfrak{m} (M)\subseteq \mathbb T (M)+\mathbb C_\mathfrak{m} (R)$ and therefore $\mathbb C_\mathfrak{m}(M)^*\subseteq \mathbb T (M)+\mathbb C_\mathfrak{m}(R)$.

To show the inverse inclusion, notice that for all $\mu \not\in \mathbb C_\mathfrak{m}(M)^*$, $(\mu +\mathbb{N}^k )\cap \mathbb C_\mathfrak{m}(M)=\emptyset$. Write $\mathbf{X}$ for the tuple of all $d$ variables and consider the first quadrant double complex $\check\Cc_\mathbf{X}^\bullet(K_\bullet ( \mathbf{X};M))$. It gives rise to a spectral sequence with first terms $K_i(\mathbf{X}; H^j_\mathfrak{m} (M))$ that abuts to $\Tor_{i-j}^R(M,S)$. As shifts in $K_i(\mathbf{X};H^j_\mathfrak{m}(M))$ have all $p$-th coordinates at most $n_p$ for any $p=1,...,k$, in degree $\mu -a$ all terms are zero because 
$K_i(\mathbf{X};H^j_\mathfrak{m} (M))_{\mu -a}$ is a sum of copies of $H^j_\mathfrak{m}(M)$ sitting in degrees $\mu +\delta$ for $\delta\in \mathbb{N}^k$. It follows that $\mu -a\not\in \mathbb T (M)$,  as claimed.

We now turn to the general case. Notice that for $u\in \mathbb Z$, the submodule $M[u]:=\{ m\in M,\ \tdeg (m)\geq u\}$ is such that $M/M[u]=H^0_{\mathfrak m}(M/M[u])$. It follows that $H^i_{\mathfrak m}(M[u])_\mu$ coincides with $H^i_{\mathfrak m}(M)_\mu $ for every $i$ and $\mu$ of total degree at least $u$. On the other hand, $K_\bullet ( \mathbf{X};M)$ and $K_\bullet ( \mathbf{X};M[u])$ coincide in total degree at least $u+d$.

Hence $\mathbb C_\mathfrak{m}(M)$ coincides with $\mathbb C_\mathfrak{m}(M[u])$ and $\hat{\mathbb T} (M)$ coincides with $\hat{\mathbb T} (M[u])$ in total degree at least  $u$.

Since the sets $E_u:=\{ t\in \mathbb Z^k,\ t_1+\cdots +t_k<u\}$  are stable under the addition of $-\mathbb N^k$, $(A\cup E_u)^*=A^*\cup E_u$ for any set $A\subseteq \mathbb Z^k$. 

It follows that
$$
\begin{array}{rlr}
\hat{\mathbb T} (M)^*&=\cap_{u\in \mathbb Z} (\hat{\mathbb T} (M)^*\cup E_{u})&(\hbox{since } \cap_{u\in \mathbb Z} E_{u} =\emptyset )\\
&=\cap_{u\in \mathbb Z} (\hat{\mathbb T} (M[u])^*\cup E_{u})&(\hbox{as } \hat{\mathbb T} (M[u])=\hat{\mathbb T} (M)\hbox{ off }E_u)\\
&=\cap_{u\in \mathbb Z} (\mathbb C_\mathfrak{m}(M[u])^*\cup E_{u})&(\hbox{by the first part})\\
&=\cap_{u\in \mathbb Z} (\mathbb C_\mathfrak{m}(M)^*\cup E_{u})&(\hbox{as }\mathbb C_\mathfrak{m} (M[u])=\mathbb C_\mathfrak{m} (M)\hbox{ off }E_{u})\\
&=\mathbb C_\mathfrak{m}(M)^*&(\hbox{since } \cap_{u\in \mathbb Z} E_{u} =\emptyset ).\\
\end{array}
$$
\end{proof}

\begin{lemma}\label{cohtornu}
Let $M$ be a graded $R$-module and $i_1,\ldots,i_p$ be distinct elements in $\{1,\ldots,k\}$. Set $T:=S[\mathbf{X}_{i_1},\ldots,\mathbf{X}_{i_p}]$ and let $\nu\in\oplus_{q\not\in \{ i_1,\ldots ,i_p\}}\mathbb{Z}e_q$. Then, for any $j\geq0$,
\begin{itemize}
    \item [i)] $H^j_{B_{i_1}+\dots+B_{i_p}}(M)_{*,\nu}\simeq H^j_{B_{i_1}+\dots+B_{i_p}}(M_{*,\nu})$ as graded $T$-modules;

\item [ii)] $\Tor^R_j(M,R/B_{i_1}+\dots+B_{i_p})_{*,\nu}\simeq\Tor^T_j(M_{*,\nu},S)$ as graded $T$-modules.
\end{itemize}
\end{lemma}

\begin{proof}
The variables in $\mathbf{X}_{i_1},\ldots,\mathbf{X}_{i_p}$ have degree zero in $\oplus_{q\not\in \{ i_1,\ldots ,i_p\}}\mathbb{Z}e_q$.
\end{proof}

A more general version of our first main result is directly derived from this lemma.

\begin{theorem}\label{C=T}
Let $M$ be a  graded $R$-module and $i_1,\dots ,i_p$ be distinct elements in $\{ 1,\dots ,k\}$. Then 
$$\hat{\mathbb T}^{i_1,\dots,i_p}(M)-\sum_{j=1}^pe_{i_j}\mathbb N = \mathbb C_{B_{i_1}+\dots+B_{i_p}}(M)-\sum_{j=1}^pe_{i_j}\mathbb N.$$
As a consequence, 
$$
\mathbb C_{B_{i_1}+\cdots +B_{i_p}}(M)^*=\hat{\mathbb T}^{i_1,\ldots ,i_p}(M)^*.
$$
\end{theorem}

\begin{proof}
According to Lemma \ref{cohtornu}, the proof follows along the same lines as in the one of Theorem \ref{basicincl}.
\end{proof}

In order to compare support with respect to several cohomologies, the following result is useful. Let $\mathcal{E}_j^{i_1,\ldots ,i_p}\subset \mathbb{N}^k$ be the set of shifts in $K_j(\mathbf{X}_{i_1},\dots ,\mathbf{X}_{i_p};R)$. 

\begin{proposition}\label{TorinCB}
Let $M$ be a graded $R$-module and $i_1,\dots ,i_p$ be distinct elements in $\{ 1,\dots ,k\}$.  For any finitely generated graded ideal $I\subseteq \sqrt{B_{i_1}+\cdots +B_{i_p}+\ann_R (M)}$,
$$\hat{\mathbb T}^{i_1,\dots,i_p}_j(M)\subseteq\bigcup_{r\leq(\sum_{l=1}^pn_{i_l})-j}\mathbb C^r_I(M)+\mathcal{E}^{i_1,...,i_p}_{j+r}+a_{i_1,\ldots ,i_p}$$ for all $j\geq0$.
\end{proposition}

\begin{proof}
Let $\check\Cc^\bullet$ be the \v{C}ech complex of $R$ with respect to a finite generating set of $I$ and consider the double complex $\check\Cc^\bullet(K_\bullet( \mathbf{X}_{i_1},\dots ,\mathbf{X}_{i_p};M))$. It gives rise to a spectral sequence with first terms $K_i(\mathbf{X}_{i_1},\dots,\mathbf{X}_{i_p}; H^j_I (M))$ that abuts to $\Tor_{i-j}^R(M,R/B_{i_1}+\dots+B_{i_p})$ since $I\subseteq \sqrt{B_{i_1}+\cdots +B_{i_s}+\ann_R (M)}$.

Now, $${\mathbb T}^{i_1,\dots,i_p}_j(M)\subseteq\bigcup_{l-r=j}\mathbb C^r_I(M)+\mathcal{E}^{i_1,...,i_p}_l=\bigcup_{r\leq(\sum_{l=1}^pn_{i_l})-j}\mathbb C^r_I(M)+\mathcal{E}^{i_1,...,i_p}_{j+r}$$ whence the result.
\end{proof}

\begin{corollary}\label{incCoh}
Let $M$ be a graded $R$-module  and $i_1,\dots ,i_p$ be distinct elements in $\{ 1,\dots ,k\}$.  
For any finitely generated graded ideal $I\subseteq B_{i_1}+\dots +B_{i_p}$,
$$
\mathbb C_{B_{i_1}+\dots +B_{i_p}}(M)\subseteq \mathbb C_I (M)-\sum_{j=1}^pe_{i_j}\mathbb{N}.$$
In particular, if $p\geq1$,
$$
\mathbb C_{B_{i_1}+\dots +B_{i_p}}(M)^*\subseteq \mathbb C_{B_{i_1}+\cdots +B_{i_{p-1}}}(M)^*.
$$
\end{corollary}

\begin{proof}
By taking union over $j$ in Proposition \ref{TorinCB} we obtain $$\hat{\mathbb T}^{i_1,\dots,i_p}(M)\subseteq \mathbb C_I(M)-\sum_{j=1}^pe_{i_j}\mathbb{N}.$$ The result follows from Theorem \ref{C=T}.
\end{proof}

Notice that this corollary implies that $ \mathbb C_{B_{i}}(M)^*\subseteq \mathbb C_B(M)^*$ for any $i$ and
$$
\mathbb C_{B_{i_1}+\dots +B_{i_p}}(M)^*\subseteq \mathbb C_{B_{i_1}}(M)^*\cap \cdots \cap \mathbb C_{B_{i_p}}(M)^*.
$$

As a consequence, we deduce the following description of  $\mathbb C_{B}(M)^*$. Notice that it is needed to stabilize the complement by adding $-\mathbb{N}^k$, see Example \ref{supportk=2ex}.

\begin{theorem}\label{cohB}
Let $M$ be a  graded $R$-module. Then, 
$$
\bigcup_i \mathbb C_{B_{i}}(M)^*= \bigcup_{i_1,\dots ,i_p} \mathbb C_{B_{i_1}+\cdots +B_{i_p}}(M)^*=\mathbb C_{B}(M)^*.
$$
\end{theorem}

\begin{proof}
First by Corollary \ref{incCoh},
$$
\bigcup_i \mathbb C_{B_{i}}(M)^*= \bigcup_{i_1,\dots ,i_p} \mathbb C_{B_{i_1}+\cdots +B_{i_p}}(M)^*\subseteq \mathbb C_{B}(M)^*.
$$
On the other hand, $\mathbb C_{B}(M)\subseteq \bigcup_{i_1,\dots ,i_p} \mathbb C_{B_{i_1}+\cdots +B_{i_p}}(M)$ by Corollary \ref{marcss} or \cite[Lemma 2.1]{BC}.
\end{proof}

\begin{remark}\label{example}
In the example of $R:=S[a,b,c,x,y,z]$  with $S$ a field ($k=2$, $n_1=n_2=3$) and $M:=R/(ax+by+cz)$, the sets $\mathbb C_{B_{i}}(M)$ and $\mathbb C_{\mathfrak m}(M)$ are stable under the addition of $-\mathbb N^2$ and independent of $S$, while $\mathbb C_{B}(M)$ is not stable under the addition of $-\mathbb N^2$ and depends upon the characteristic of $S$ (two distinct characteristics give rise to two distinct supports). It shows that the determination of $\mathbb C_{B}(M)$ could reserve some delicate points (for this example, we could only fully determine $\mathbb C_{B}(M)$ in characteristic zero or two).
\end{remark}

By Theorem \ref{C=T}, the first equality in Theorem \ref{cohB} follows from an inclusion of support of Tor modules that could be proved directly, as we now show.

\begin{proposition}
Let $M$ be a graded $R$-module and $i_1,\dots ,i_p,j_1,\dots ,j_q$ be distinct elements in $\{ 1,\ldots ,k\}$, then
$$
{\mathbb T}^{i_1,\dots ,i_p,j_1,\dots ,j_q}_j (M)\subseteq \bigcup_{l\leq j}{\mathbb T}^{i_1,\dots ,i_{p}}_{j-l} (M)+\mathcal{E}_l^{j_1,\dots ,j_q}
$$ for all $j\geq0$.
\end{proposition}

\begin{proof}
The double complex $K_\bullet (\mathbf{X}_{j_1},\dots ,\mathbf{X}_{j_q};K_\bullet (\mathbf{X}_{i_1},\dots ,\mathbf{X}_{i_p};M))
$ gives rise to a spectral sequence with first terms  
$$
K_\bullet (\mathbf{X}_{j_1},\dots ,\mathbf{X}_{j_q};\Tor_\bullet^{R}(M,R/B_{i_1}+\ldots+B_{i_p}))
$$
that abuts to a filtration of $\Tor_\bullet ^{R}(M,R/B_{i_1}+\ldots+B_{i_p}+B_{j_1}+\ldots+B_{j_q})$.
\end{proof}

\begin{corollary}\label{hatT}
Let $M$ be a graded $R$-module and $i_1,\dots ,i_p,j_1,\dots ,j_q$ be distinct elements in $\{ 1,\ldots ,k\}$, then
$$
\hat{\mathbb T}^{i_1,\dots ,i_p,j_1,\dots ,j_q} (M)^*\subseteq \hat{\mathbb T}^{i_1,\dots ,i_p} (M)^*.
$$
\end{corollary}

\section{Some consequences on the truncation of modules}

We first recall two results in the classical case where $R=S[X_1,\ldots,X_n]$ is a standard $\mathbb{Z}$-graded polynomial ring and $\mathfrak{m}=(X_1,\ldots ,X_n)$.

\begin{proposition}\label{regbound}
Let $C_\bullet$ be a graded complex of $R$-modules such that  $H^p_{\mathfrak m}(H_i(C_\bullet))=0$ for any $i\not= 0$ and $p>i$. Then,
$$\reg(H_0(C_\bullet))\leq\max_i \{\reg(C_i)-i\}.$$
\end{proposition}

\begin{proof}
It follows along the same lines as for instance in the proof of \cite[Lemma 2.2]{CFN} by consideration of the two spectral sequences from $\check\Cc^\bullet_{\mathbf X} (C_\bullet)$ (the hypothesis  $C_i=0$ for $i<0$ there is not used except to restrict the $\max$ to positive values of $i$). 
\end{proof}

\begin{lemma}\label{standardsupport}
Let $M$ be a graded $R$ module and $t$ an integer.Then,
\begin{itemize}
    \item [(i)] $\mathbb C^0_\mathfrak{m}(M_{t+\mathbb N})=\mathbb C^0_\mathfrak{m}(M)\cap(t+\mathbb N)$;
    \item [(ii)] $\mathbb C^1_\mathfrak{m}(M_{t+\mathbb N})=\mathbb C_\mathfrak{m}^1(M)\cup\{\Supp_{\mathbb Z}(M/H^0_\mathfrak{m}(M))\cap(t-1-\mathbb N)\}$;
    \item [(iii)] $\mathbb C^j_\mathfrak{m}(M_{t+\mathbb N})=\mathbb C^j_\mathfrak{m}(M)$ for all $j\geq2$.
\end{itemize}
\end{lemma}

\begin{proof}
It directly derives from the exact sequence
$$
\xymatrix@=1em{
0\ar[r] & H^0_\mathfrak{m}(M_{t+\mathbb N})\ar[r] & H^0_\mathfrak{m}(M)\ar[r] & M/M_{t+\mathbb N}\ar[r] & H^1_\mathfrak{m}(M_{t+\mathbb N})\ar[r] & H^1_\mathfrak{m}(M)\ar[r] & 0}$$ and the isomorphisms $H^j_\mathfrak{m}(M_{t+\mathbb N})\simeq H^j_\mathfrak{m}(M)$ for all $j\geq2$, that are in turn provided by the short exact sequence 
$$
\xymatrix@=1em{
0\ar[r] & M_{t+\mathbb N}\ar[r]&M\ar[r]&M/M_{t+\mathbb N}\ar[r]&0\\}
$$
and the vanishing of $H^i_\mathfrak{m}(M/M_{t+\mathbb N})$ for $i>0$.
\end{proof}

We now return to our general setting of a standard multigraded ring $R$.

\begin{definition} 
Let $\mathbb E_i:=\{ (q_1,\ldots ,q_k)\in \mathbb Z^k\ \vert\ q_i\geq 0\}$ and
$$
\mathbb E_{i_1,\ldots ,i_p} :=\mathbb E_{i_1}\cap\cdots\cap E_{i_p}.
$$
\end{definition}
Notice that $\mathbb E_{1,\ldots ,k}=\mathbb N^k$. With these notations,

\begin{proposition}\label{cohtrunc1}
Let $M$ be a graded $R$-module, $i\in \{ 1,\ldots ,k\}$ and $t\in \mathbb Z^k$. Then 
\begin{itemize}
   \item [(i)] For $\mu\in t+\mathbb N^k$,
    $$
     H^j_{B_i}(M_{ t+\mathbb N^k})_\mu = H^j_{B_i}(M)_\mu ,\quad \forall j.
    $$

    \item [(ii)]  For $\mu\not\in t+\mathbb N^k$,
\begin{itemize}    
    \item [(1)] If $\mu\not\in t+\mathbb E_{1,\ldots,\widehat{i,}\ldots,k}$, $    H^j_{B_i}(M_{t+\mathbb N^k})_\mu =0$ for all $j$;
    \item [(2)] If $\mu\in t+\mathbb E_{1,\ldots,\widehat{i,}\ldots,k}$, $H^0_{B_i}(M_{t+\mathbb N^k})_\mu =0$, the sequence
$$
\xymatrix@=1em{
0\ar[r] &  (M/H^0_{B_i}(M))_\mu\ar[r] & H^1_{B_i}(M_{t+\mathbb N})_\mu \ar[r] & H^1_{B_i}(M)_\mu \ar[r] & 0}$$ is exact    and $H^j_{B_i}(M_{t+\mathbb N^k})=H^j_{B_i}(M)$ for all $j\geq2$.
\end{itemize}
\end{itemize}
\end{proposition}

\begin{proof}
According to Lemma \ref{cohtornu} i), this follows from Lemma \ref{standardsupport} applied to the ring $T=S[\mathbf X_i]$.
\end{proof}

Set $\mathbf 1 :=(1,\ldots ,1)=\sum_{i=1}^k e_i\in \mathbb N^k$.
\begin{corollary}\label{cohtrunc2}
Let $M$ be a graded  $R$-module and $t\in \mathbb Z^k$. Then,
$$
\mathbb C_{\mathfrak m}(M_{ t+\mathbb N^k})^*\subseteq \bigcap_i (\mathbb C_{B_{i}}(M)\cap (t+\mathbb N^k))^* \cup (t-\mathbf 1 -\mathbb E_i) .
$$
\end{corollary}

\begin{proof}
By Proposition \ref{cohtrunc1}, 
$$
\mathbb C_{B_i}(M_{ t+\mathbb N^k})\subseteq (\mathbb C_{B_i}(M)\cap (t+\mathbb N^k))\cup (t+(\mathbb E_{1,\ldots,\widehat{i,}\ldots,k}\setminus \mathbb N^k)).
$$
As $(\mathbb E_{1,\ldots,\widehat{i,}\ldots,k}\setminus \mathbb N^k )^* =-\mathbb E_{i}-\mathbf 1$, the conclusion follows from Corollary \ref{incCoh}.
\end{proof}

\begin{proposition}\label{tortrunc}
Let $M$ be a graded $R$-module 
$
\Delta :=[0,n_1-1]\times\cdots\times [0,n_k-1]\subseteq \mathbb N^k
$
and $t\in \mathbb Z^k$. Then
$$
\mathbb T (M_{ t+\mathbb N^k})\subseteq (-a+(\mathbb C_B(M)^*\cap (t+\mathbb N^k))) \cup (t+\Delta ).
$$
In particular, $\mathbb T (M_{ t+\mathbb N^k})\subseteq t+\Delta$ if $t\not\in \mathbb C_B(M)^*$.
\end{proposition}

\begin{proof}
First, $\hat{\mathbb T} (M_{ t+\mathbb N^k})\subseteq \hat{\mathbb T} (M_{ t+\mathbb N^k})^*=\mathbb C_{\mathfrak m}(M_{ t+\mathbb N^k})^*$ by Theorem \ref{basicincl}. By Theorem \ref{cohB}, $\mathbb C_B(M)^*=\cup_i \mathbb C_{B_i}(M)^*$, hence 
$$
\mathbb T (M_{ t+\mathbb N^k})+a\subseteq (\mathbb C_B(M)\cap  (t+\mathbb N^k))^* \cup (t-\mathbf 1 -\mathbb N^k).
$$

by Corollary \ref{cohtrunc2}. But $\mathbb T (M_{ t+\mathbb N^k})\subseteq \Supp (M_{ t+\mathbb N^k})\subseteq t+\mathbb N ^k$ and 
$$ (-a-\mathbf 1 -\mathbb N^k)\cap \mathbb N ^k =\Delta.$$

If $t\not\in \mathbb C_B(M)^*$, then $\mathbb C_B(M)^*\cap (t+\mathbb N^k) =\emptyset$ by Lemma \ref{stabcomp}.
 \end{proof}

Every $\mathbb Z^k$-graded module $M$ is also a $\mathbb Z$-graded module, for the total degree grading, and we set $\reg (M)$ for its Castelnuovo-Mumford regularity. For $\mu =(\mu_1,\ldots ,\mu_k )\in \mathbb Z^k$, write $\mu^+:=(\max \{ \mu_1 ,0\},\ldots ,\max \{ \mu_k ,0\} )$ and $\vert \mu \vert :=\mu_1 +\cdots +\mu_k$.

The next lemma follows from \cite[the proof of 1.7]{EES}; it provides a criterion for the truncation of a module $M_{t+\mathbb N^k}$ with $t\in \mathbb Z ^k$ to satisfy $\reg (M)=\vert t\vert$ (if the base ring $S$ is a field, usual terminology is that $M$ is generated in degree $\vert t\vert$ and has a linear resolution).

\begin{lemma}\label{slinrestrunc}
Let $M$ be a graded $R$-module such that $\tdeg (m)\geq u$ for some $u\in\mathbb Z$ and every $m\in M$. For $t\in \mathbb Z^k$, let
$$
\delta_i :=\sup_{\mu \in \mathbb T_i(M)} \{ \vert (t-\mu )^{+}\vert - \vert t - \mu \vert\}.
$$
Then $\reg (M_{t+\mathbb N^k})\leq\vert t\vert +\max_i \{ \delta_i -i\}$.

In particular, unless $M_{t+\mathbb N^k}=0$, if $\vert (t-\mu )^{+}\vert \leq \vert t - \mu \vert +i$ for any $i$ and $\mu \in \mathbb T_i (M)$,
$$
\reg (M_{t+\mathbb N^k})=\vert t\vert .
$$
\end{lemma}

\begin{proof}
As noticed by Eisenbud, Erman and Schreyer, any truncation of $R$ has a linear resolution (given by the tensor product of some Eagon-Northcott complexes). As a consequence, considering $R$ with its  standard $\mathbb Z$-grading 
$$
\reg (R_{t+\mathbb N^k})=\indeg (R_{t+\mathbb N^k})=\vert t^{+}\vert .
$$ 
Hence $\reg (R(-\mu )_{t+\mathbb N^k})=\vert \mu \vert + \vert (t-\mu)^{+}\vert $.
As   $\tdeg (m)\geq u$ for every $m\in M$, $M$ admits a graded free $R$-resolution $F_\bullet$ with $F_i=\oplus_j R(-t_{i,j})$ satisfying $t_{i,j}\in \cup_{p\leq i}\mathbb T_p (M)$,  by Lemma \ref{TorFR}. The complex $(F_\bullet)_{t+\mathbb N^k}$ resolves $M_{t+\mathbb N^k}$ and for the standard grading of $R$, 
$$
\reg ((F_i)_{t+\mathbb N^k})-\vert t\vert =\sup_{j}\{ \vert t_{i,j}\vert +\vert (t- t_{i,j})^{+}\vert \} -\vert t\vert \leq \max_{p\leq i}\{ \delta_p\} .
$$
The conclusion follows by Proposition \ref{regbound}.
\end{proof}

\begin{remark}\label{remarktrunc}
To apply the above results in order to find $t$ such that $\reg (M_{t+\mathbb N^k})=\vert t\vert$, it is useful to have in mind that :

(i) $\delta_i=0$ if and only if $t\geq \mu$ (i.e. $t_j\geq \mu_j$, $\forall j$) for every $\mu\in \mathbb T_i$,

(ii) one may replace $M$ by $M'$ such that $M_{\mu+\mathbb N^k}=M'_{\mu+\mathbb N^k}$ and choose $t\in  \mu +\mathbb N^k$,

(iii) for any $\mu\in \mathbb Z^k$, $N:=M/M_{\mu +\mathbb N^k}$ satisfies $N=H^0_B(N)$ and therefore 
$$
H^i_B(M)_\nu =H^i_B(M_{\mu +\mathbb N^k})_\nu ,\quad \forall \nu\in \mu+\mathbb N^k.
$$
\end{remark}

\begin{proposition}\label{linrestrunc}
Let $M$ be a graded $R$-module and $\mu\not\in \mathbb C_B(M)^*$. 

Then for any $t\in \mu +(n_1-1,\ldots ,n_k-1)+\mathbb N^k$, $M_{t +\mathbb N^k}$ has a regularity $\vert t\vert$.
\end{proposition}

\begin{proof}
It follows from Proposition \ref{tortrunc} and Lemma \ref{slinrestrunc}, according to Remark \ref{remarktrunc}.
\end{proof}

For $M=R$, $\reg (R_{t+\mathbb N^k})=\vert t\vert$ if and only if $t\in \mathbb N^k$ and $\mathbb C_B(R)^*+(n_1-1,\ldots ,n_k-1)=\mathbb Z^k\setminus \mathbb N^k$.\\

\noindent{\bf Acknowledgements.} The second-named author was supported by a CAPES Doctoral Scholarship.

\end{document}